\documentclass[preprint]{elsarticle}

\usepackage{xurl} 
\usepackage{hyperref}

\usepackage{lineno,hyperref,amsmath, amssymb,pdflscape,psfrag,setspace,commath,bbm}
\usepackage{graphicx}
\usepackage{enumerate}
\usepackage{mathrsfs}
\usepackage{bm}
\usepackage{longtable}
\usepackage{tabularx}
\usepackage{float} 
\floatstyle{plaintop}%
\restylefloat{table}%
\usepackage{float} 
\usepackage{multirow} 
\usepackage{colortbl} 
\usepackage[bf]{caption}
\usepackage{multirow,lipsum,hyphenat}
\usepackage{subfig}
\usepackage{multicol}
\usepackage{makecell}
\usepackage{appendix}
\usepackage{amsmath}

\usepackage{pgfplots}
\pgfplotsset{compat=1.18}

\usepgfplotslibrary{fillbetween}
\usepackage{booktabs,tabularx,multirow}
\usepackage[margin=1in]{geometry}
\usepackage{tikz}
\usetikzlibrary{arrows.meta, positioning}
\usetikzlibrary{calc}
\usetikzlibrary{shapes.geometric}
\usepackage{xcolor}
\definecolor{blue}{rgb}{0.0, 0.0, 1.0}
\definecolor{bleudefrance}{rgb}{0.19, 0.55, 0.91}
\definecolor{deepskyblue}{rgb}{0.0, 0.75, 1.0}
\definecolor{green}{rgb}{0.0, 0.5, 0.0}
\hypersetup{colorlinks,linkcolor={blue},citecolor={green},urlcolor={black}}
\usepackage{todonotes}
\newcommand{\TG}[1]{\todo[linecolor=orange, backgroundcolor=orange!20, tickmarkheight=1ex]{TG: #1}}

\newtheorem{theorem}{Theorem}

\newproof{proof}{Proof}

\DeclareMathOperator{\re}{\mathbbm{R}}

\modulolinenumbers[5]
\begin{document}

\begin{frontmatter}
\title{A Mathematical Model of Dengue Transmission Incorporating Hospital Capacity and Threshold-Based Fogging Interventions}

\author[IND]{Dipo Aldila}
\ead{aldiladipo@sci.ui.ac.id}
\address[IND]{Department of Mathematics, Faculty of Mathematics and Natural Sciences, Universitas Indonesia, Depok 16424, Indonesia.}
\author[ESPOL,CFD]{Joseph P\'aez Ch\'avez}
\ead{jpaez@espol.edu.ec}
\address[ESPOL]{Center for Applied Dynamical Systems and Computational Methods (CADSCOM), Faculty of Natural Sciences and Mathematics, Escuela Superior Polit\'ecnica del Litoral, P.O. Box 09-01-5863, Guayaquil, Ecuador}
\address[CFD]{Center for Dynamics, Department of Mathematics, TU Dresden, D-01062 Dresden, Germany.}
\author[UOrdu]{Aytül Gökçe}
\ead{aytulgokce@odu.edu.tr}
\address[UOrdu]{Department of Mathematics, Faculty of Science and Letters, Ordu University, 52200 Ordu, Türkiye}
\author[UKO]{Thomas Götz}
\ead{goetz@uni-koblenz.de}
\address[UKO]{Mathematical Insitute, University of Koblenz, 56070 Koblenz, Germany}
\author[JGU,IQCB]{Burcu Gürbüz$^*$}
\cortext[mycorrespondingauthor]{Corresponding author}
\ead{burcu.gurbuz@uni-mainz.de}
\address[JGU]{Institute of Mathematics, Johannes--Gutenberg Universität Mainz, 55099 Mainz, Germany}
\address[IQCB]{Institute for Quantitative and Computational Biosciences (IQCB), Johannes Gutenberg University Mainz, 55128 Mainz, Germany}

\begin{abstract}
Dengue remains an annual public health problem in many tropical countries, and recurring outbreaks indicate that current intervention strategies are not yet fully effective. Fumigation (fogging) and hospitalization are two of the most widely used interventions, but existing mathematical models typically assume unlimited hospital capacity and continuously applied fogging, overlooking the practical constraints that shape real dengue control programs. In this paper, we develop a non-smooth ordinary differential equation model of dengue transmission that incorporates a finite hospital capacity and a threshold-triggered fogging policy, activated once reported infections reach a specified fraction of that capacity. The model distinguishes three operating regimes: no fogging with full hospitalization when infections are low, fogging with full hospitalization once the activation threshold is reached, and fogging with only partial hospitalization once capacity is exceeded, in which case untreated individuals remain infectious and continue to contribute to transmission. We establish the existence and local stability of the disease-free and endemic equilibria. Numerical continuation confirms these analytical results and further reveals boundary-equilibrium bifurcations precisely at the switching thresholds, together with a Hopf bifurcation once infections exceed hospital capacity, giving rise to sustained periodic outbreaks. We also identify a fold bifurcation near the epidemic threshold that produces additional unstable equilibria even when the disease would otherwise be expected to die out, indicating a more intricate equilibrium structure than typically assumed. Continuation of the periodic solution with respect to the fogging rate and activation threshold reveals locally optimal values that minimize peak infections and the fraction of asymptomatic cases, while overly early or intense fogging increases cost without proportional epidemiological benefit. These results show that hospital capacity, fogging intensity, and the activation threshold jointly determine whether the disease settles into a controllable endemic state or transitions into a recurring outbreak regime, providing quantitative guidance for the design of reactive dengue control policies.
\end{abstract}

\end{frontmatter}


\section{Introduction}

Dengue is a mosquito-borne viral disease caused by four antigenically distinct serotypes of the dengue virus. It is primarily transmitted by infected female \textit{Aedes aegypti} mosquitoes. Its clinical manifestations range from asymptomatic infection and self-limiting febrile illness to severe dengue, which may involve plasma leakage, hemorrhage, organ impairment, shock, and death \cite{wilder2019dengue,PazBailey2024Dengue}. Over recent decades, dengue has expanded substantially in geographical range and incidence, driven by rapid urbanization, population mobility, climatic suitability, and the increasing distribution of competent mosquito vectors. Bhatt et al.\ estimated that there are approximately 390 million dengue infections per year, of which nearly 96 million manifest clinically \cite{bhatt2013global}. Subsequent global burden studies have confirmed that dengue causes considerable morbidity, mortality, and economic losses, particularly in tropical and subtropical regions \cite{stanaway2016global,shepard2016global}. Recently, the World Health Organization reported an unprecedented level of global transmission in 2024, with over 14 million reported cases, more than 52,000 severe cases, and over 11,000 deaths. These epidemiological trends emphasize the need for intervention strategies that effectively reduce transmission and are feasible under limited public health resources \cite{mondiale2025dengue,guzman2010dengue}.

{The clinical management of dengue primarily relies on early diagnosis, careful monitoring, and timely hospitalization of patients who exhibit warning signs or develop severe symptoms, as there is currently no specific antiviral treatment available \cite{wilder2019dengue,PazBailey2024Dengue}. Appropriate hospital-based case management can substantially reduce dengue mortality. However, large outbreaks can rapidly increase the demand for hospital beds, medical personnel, diagnostic services, and supportive treatment. When the number of patients requiring admission approaches or exceeds the hospital's capacity, delays in hospitalization and treatment can lead to worse outcomes. Therefore, hospital capacity should not be considered an unlimited external resource but rather a dynamic constraint that interacts with epidemic progression. Abdelrazek et al. incorporated limited medical resources into a dengue transmission model through a nonlinear recovery function. They demonstrated that resource saturation can generate multiple endemic equilibria, backward bifurcation, and oscillatory dynamics \cite{abdelrazec2016modeling}. Aldila et al. further demonstrated that, together with case detection and public awareness, hospital capacity can strongly influence prospects for dengue elimination \cite{aldila2023impact}. More generally, recent epidemic models with finite healthcare facilities show that the availability of hospital beds can modify both disease thresholds and qualitative dynamics. This supports the explicit inclusion of capacity-dependent hospitalization in transmission models \cite{misra2026multiobjective}. The latter study motivates the inclusion of the healthcare capacity component in the present framework and models admission as a function of infection prevalence and the number of unoccupied hospital beds.}

{Dengue is transmitted by mosquitoes, so controlling these mosquitoes is an important part of responding to outbreaks. Common measures include reducing sources, controlling larvae, eliminating water-holding containers, using biological controls, and applying adulticides through indoor residual spraying or space spraying, commonly referred to as fogging. Although fogging can rapidly reduce the adult mosquito population, its epidemiological impact depends on the timing, intensity, spatial coverage, mosquito susceptibility to insecticides, and frequency of application. Using a seasonally forced dengue model, Oki et al.\ showed that fogging is most effective when implemented between the beginning of the wet season and the epidemic prevalence peak. Delayed application, however, may have only a limited impact \cite{oki2011optimal}. Similarly, mathematical optimal-control studies have shown that appropriately timed insecticide application can reduce infected mosquito and human populations, though the required intensity of intervention may entail substantial operational costs \cite{rodrigues2010insecticide,rodrigues2013bioeconomic,fitria2017optimal}. These findings suggest that fogging should not merely be represented as a constant increase in mosquito mortality. Rather, its implementation should reflect the epidemiological information available to public health authorities.}

In practice, government responses to dengue are generally state-dependent. Intensive fogging campaigns are rarely maintained continuously due to financial constraints, logistical limitations, environmental concerns, and the risk of insecticide resistance. Instead, public health authorities typically intensify mosquito control when surveillance indicators, such as reported infections, hospitalized cases, severe cases, or mosquito indices, cross a prescribed alert level. The intervention is relaxed or discontinued when these indicators decline sufficiently. Thus, the control policy depends on the current state of the epidemic rather than only calendar time. De Lara and Sepúlveda-Salcedo formulated dengue vector control as a viability problem, selecting the fogging rate as a feedback function of the current proportions of infected humans and mosquitoes to maintain infection prevalence below a prescribed threshold \cite{de2016viable}. Such feedback formulations provide an important mathematical foundation for incidence-triggered interventions. However, a policy based on an explicit activation threshold can produce abrupt changes in the governing vector field and dynamical behavior that models employing only smooth, continuously varying control functions cannot adequately represent. Mathematical models have long been used to identify the mechanisms underlying dengue persistence and to compare alternative control measures. Classical host-vector models typically divide the human population into four classes: susceptible, exposed, infectious, and recovered. Mosquitoes are represented by three compartments: susceptible, exposed, and infectious. These models have been extended to incorporate seasonality, multiple serotypes, temporary cross-immunity, antibody-dependent enhancement, vaccination, human mobility, spatial heterogeneity, behavioral responses, and optimal vector control. The structural diversity of these models has been reviewed by Andraud et al.\ and more recently by Ogunlade et al.\ \cite{andraud2012dynamic,ogunlade2023systematic}. Although a substantial theoretical basis for dengue epidemiology is provided by this literature, hospitalization constraints and reactive vector-control policies are often investigated separately. Consequently, models with conventional smooth controls may overlook the interaction between finite hospital capacity and the discontinuous implementation of emergency mosquito control measures.

Motivated by these considerations, the present study develops a non-smooth mathematical model of dengue transmission that incorporates limited hospital capacity and a threshold-triggered fogging intervention. Hospital admission is limited by the number of available beds, so that epidemic growth directly affects access to treatment; conversely, healthcare saturation feeds back into disease progression by leaving a fraction of infected individuals untreated and hence infectious. Fogging is activated once the number of hospitalized infections reaches a predefined threshold, reflecting a public health policy that responds to the epidemic situation rather than one that is permanently applied. This framework allows us to examine how hospital saturation, the intervention threshold, and the fogging intensity jointly influence disease persistence, outbreak magnitude, and the frequency of public health interventions, and it provides a basis for identifying threshold values that prevent healthcare overload while avoiding unnecessarily frequent or prolonged fogging campaigns.

The remainder of this paper is organized as follows. Section~\ref{Sec2} presents the construction of the mathematical model, describing the underlying assumptions and the parameterization of the model based on existing literature. Section~\ref{Sec3} provides the analytical results, including the existence and stability analysis of the equilibria in each of operating regimes. Section~\ref{Sec4} presents the numerical investigation, employing continuation techniques to construct bifurcation diagrams and to explore the emergence and characteristics of periodic solutions. Finally, Section~\ref{Sec5} discusses the public health implications of these findings, along with the limitations of the present study and directions for future research.

\section{Model Formulation and Parameterization}\label{Sec2}
\subsection{The Mathematical Model}
As mentioned before, mathematical modeling has been widely used by many researchers to illustrate the dynamics of dengue transmission among human and mosquito populations. In general, such models divide the human and mosquito populations based on their health status or on the interventions incorporated into the model. Here, we present a host--vector model for dengue transmission, where the human population is denoted by $N_h$ and the mosquito population by $N_m$. We group the total human population into four compartments, namely the susceptible population $S(t)$, infected without symptoms (asymptomatic) population $A(t)$, infected with symptoms (symptomatic) population $I(t)$, and recovered population $R(t)$. Hence, we have
\[
N_h(t) = S(t) + A(t) + I(t) + R(t).
\]
On the other hand, the mosquito population is divided into only two compartments, namely the susceptible mosquito population $U(t)$ and the infected mosquito population $V(t)$. Hence, we have
\[
N_m(t) = U(t) + V(t).
\]
\newpage
\begin{figure}
\begin{tikzpicture}[
    x=1cm,y=1cm,
    >=Stealth,
    every node/.style={font=\fontsize{13}{15}\selectfont},
    comp/.style={
        draw={rgb,255:red,6; green,40; blue,58},
        line width=1.1pt,
        rounded corners=10pt,
        minimum width=3.1cm,
        minimum height=1.35cm,
        align=center,
        fill=white
    },
    purp/.style={
        draw={rgb,255:red,160; green,43; blue,147},
        line width=1.2pt,
        -{Stealth[length=7pt,width=7pt]}
    },
    blueDashed/.style={
        draw={rgb,255:red,25; green,101; blue,140},
        dashed,
        line width=1.1pt,
        -{Stealth[length=7pt,width=7pt]}
    },
    bigblack/.style={
        draw=black,
        fill=black,
        line width=1.2pt,
        -{Triangle[length=12pt,width=14pt]}
    }
]

\fill[gray!10] (-1.2,-2.7) rectangle (15.5,10.2);

\node[comp] (S) at (2.2,6.0) {Susceptible\\human};

\node[comp] (A) at (7.9,8.25) {Asymptomatic\\human};

\node[comp] (Y) at (7.9,4.25) {Symptomatic\\human};

\node[comp] (R) at (13.0,6.0) {Recovered\\human};

\node[comp] (SM) at (4.8,-1.7) {Susceptible\\mosquitoes};
\node[comp] (IM) at (11.1,-1.7) {Infected\\mosquitoes};

\draw[purp] (S.north) -- ++(0,2.0) -| (A.west);
\node[align=center] at (4.6,9.4) {Catch dengue,\\no symptoms};

\draw[purp] (S.south) -- ++(0,-1.0) -| (Y.west);
\node[align=center] at (4.55,2.95) {Catch dengue,\\show symptoms};

\draw[purp] (A.south) -- (Y.north);
\node[align=center] at (5.95,6.4) {Develop symptoms,\\hospitalized,\\or home\\treated};

\draw[purp] (A.east) -- ++(3.8,0) |- (R.north);
\node[align=center] at (11.9,8.8) {Recover\\naturally};

\draw[purp] (Y.east) -- ++(3.4,0) |- (R.south);
\node[align=center] at (11.05,3.0) {Recover due\\to treatment};

\draw[blueDashed] (-0.2,6.0) -- (S.west);
\node[anchor=west] at (-1.5,6.2) {Newborn};

\draw[blueDashed] ($(S.north west)+(0.18,-0.02)$) -- ++(-0.95,0.78);
\node[align=center] at (0.9,7.45) {Natural\\death};

\draw[blueDashed] (A.north) -- ++(0,1.15);
\node[align=center] at (8.85,9.7) {Natural\\death};

\draw[blueDashed] (R.east) -- ++(1.2,0);
\node[align=center] at (15.1,6.65) {Natural\\death};

\draw[blueDashed] ($(Y.south)+(-0.42,0)$) -- ++(0,-1.0);
\draw[blueDashed] ($(Y.south)+(0.55,0)$) -- ++(0,-1.0);
\node[align=center] at (6.95,2.15) {Natural\\death};
\node[align=center] at (8.95,2.25) {Diseased};

\node[
    ellipse,
    fill=red!90!black,
    text=white,
    minimum width=2.2cm,
    minimum height=1.2cm,
    align=center,
    font=\fontsize{12}{14}\selectfont
] (D) at (8.05,0) {Dengue\\cycles};

\draw[bigblack] (6.95,-1.15)
    .. controls (6.1,0.2) and (5.95,1.8) ..
    (6.15,2.55);

\draw[bigblack] (9.15,2.55)
    .. controls (9.15,1.55) and (10.25,0.0) ..
    (9.55,-1.05);

\draw[blueDashed] (1.6,-1.7) -- (SM.west);
\node[anchor=west] at (1.15,-1.35) {Newborn};

\draw[purp] (SM.east) -- (IM.west);
\node[align=center] at (7.95,-1.5) {Catch dengue};

\draw[blueDashed] ($(SM.south)+(-0.55,0)$) -- ++(0,-1.05);
\draw[blueDashed] ($(SM.south)+(0.55,0)$) -- ++(0,-1.05);
\node[align=center] at (3.85,-4) {Natural\\death};
\node[align=center] at (5.95,-4) {Fogging};

\draw[blueDashed] ($(IM.south)+(-0.55,0)$) -- ++(0,-1.05);
\draw[blueDashed] ($(IM.south)+(0.55,0)$) -- ++(0,-1.05);
\node[align=center] at (10.15,-4) {Natural\\death};
\node[align=center] at (12.15,-4) {Fogging};

\end{tikzpicture}
\caption{The transmission diagram of the model}\label{fig1}
\end{figure}
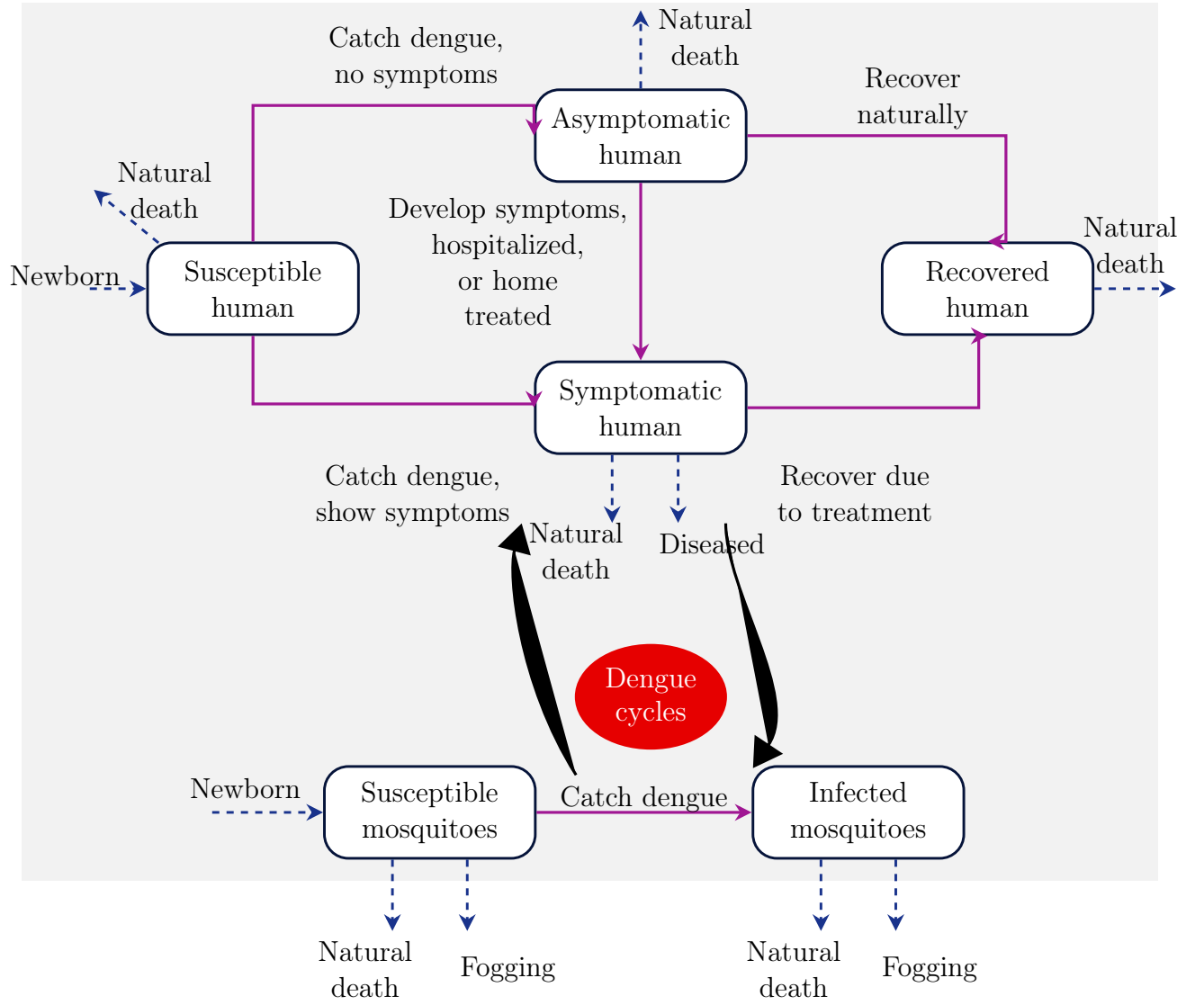

Our model follows the transmission diagram given in Fig.\ref{fig1} and is formulated under the following assumptions. First, there is no vertical transmission of the dengue virus in either the human or mosquito population. Additionally, there is no migration in either population. Both humans and mosquitoes are assumed to be born susceptible to dengue. We also assume homogeneous mixing between humans and mosquitoes, meaning that the vector has an equal probability of infecting any human individual. Furthermore, dengue infection in humans does not always present symptoms. Beyond the transmission dynamical properties, our model also considers the constraints of healthcare resources, which in this case are represented by the maximum hospital capacity, denoted by $C$. These limitations are expected to influence the disease outcomes, which will be accommodated in the infection process and will be discussed later in this section.

With these assumptions, the model is constructed as follows. The human population increases due to a constant natural birth rate, denoted by $\Lambda_h$, with all newborns entering the susceptible compartment. Susceptible individuals become infected with dengue after being bitten by infected mosquitoes $V$ at a constant infection rate $\beta_h$. We assume that newly infected individuals do not always develop dengue symptoms. A proportion $p$ of new infections will become asymptomatic individuals $A$, while the remaining proportion $1-p$ will directly move to the symptomatic compartment $I$.

Next, we assume that $\alpha$ represents the transition rate from the asymptomatic compartment $A$ to the symptomatic compartment $I$, due to the development of symptoms.

The natural recovery rate of asymptomatic individuals, driven by the immune system, is denoted by $\gamma_0$. On the other hand, individuals with symptoms require specific treatment, either in a hospital or through home-based care. We assume that the maximum hospital capacity is given by $C$. Hence, symptomatic individuals can receive hospital treatment only if the number of symptomatic cases is below this maximum capacity. If the number of symptomatic individuals exceeds the hospital capacity, they are assumed to undergo self-treatment at home. To accommodate this mechanism, we define the following function to model the recovery rate of symptomatic individuals:
\begin{equation}\label{recovery}
    f(I) =
    \begin{cases}
      \gamma_1 I, & \text{if } I < C, \\
      \gamma_1 C + \gamma_2 (I - C), & \text{if } I \geq C,
    \end{cases}
\end{equation}
where $\gamma_1$ represents the recovery rate due to hospitalization, while $\gamma_2$ denotes the recovery rate from home-based treatment, with $\gamma_1 > \gamma_2$. All recovered individuals move to the recovered compartment. Additionally, we assume that dengue may cause death at a constant rate $\delta$ only for symptomatic individuals. For the sake of analytic tractability, we assume this death rate to be the same for hospitalized or home-based individuals.

Next we will model the dynamic of the mosquitoes population. The mosquito population increases due to a constant birth rate, denoted by $\Lambda_m$, with all newborns entering the susceptible mosquito compartment $U$. Infection in the mosquito population occurs when susceptible mosquitoes $U$ bite asymptomatic infected individuals or symptomatic individuals who are not receiving hospital treatment, since we assume that symptomatic individuals who receiving hospital treatment are unable to transmit the dengue virus due to the clean environment in the hospital. To model this mechanism, we define the infection term in the mosquito population as follows:
\begin{equation}\label{inf}
    g(U,A,I) =
    \begin{cases}
      \beta_{m1} U A, & \text{if } I < C, \\
      \beta_{m1} U A + \beta_{m2} U (I - C), & \text{if } I \geq C,
    \end{cases}
\end{equation}
where $\beta_{m1}$ and $\beta_{m2}$ represent the infection rates from asymptomatic individuals and symptomatic individuals who are not hospitalized, respectively. Due to a typically higher viral load of symptomatic individuals, we assume $\beta_{m2}>\beta_{m1}$.

Another important feature of our model is the inclusion of fogging interventions to reduce the mosquito population. We assume that this fogging intervention is implemented when early warning signs of a dengue outbreak are present, which in our model are indicated by a minimum number of infected individuals being treated in the hospital. To model this fogging strategy, denoted by $h(I)$, we define the following fogging intervention:
\begin{equation}\label{fog}
    h(I) =
    \begin{cases}
     0, & \text{if } I < kC, \\
     \eta, & \text{if } I \geq  kC,
    \end{cases}
\end{equation}
where $\eta$ denotes the fogging rate and $k \in (0,1]$ represents the minimum percentage threshold of treated individuals in the hospital, which serves as an early warning indicator for dengue endemicity. For epidemiological interpretation, $kC$ represents an operational early warning indicator in the field, which reflects the level of hospital burden that triggers the fogging strategy by the authorities in a potential endemic area. When the number of hospitalized cases exceeds this threshold, fogging is activated rapidly to suppress the mosquito population. Therefore, a smaller value of $k$ indicates that public health authorities respond more quickly, since only a small number of hospitalized individuals is required to initiate the fogging intervention. On the other hand, a higher value of $k$ represents a situation in which the authorities respond relatively late, requiring a larger number of hospitalized individuals before the fogging intervention is put into place.

With the above model description, the mathematical model of dengue transmission, incorporating the number of hospitalized individuals and fogging interventions, is given by the following system of ordinary differential equations:
\begin{equation}\label{model}
    \begin{aligned}
        \dot{S} &= \Lambda_h - \beta_h S V - \mu_h S,  \\
    \dot{A} &= p \beta_h S V - \left(\alpha + \gamma_0 + \mu_h\right)A,  \\
    \dot{I} &= (1-p)\beta_h S V + \alpha A - f(I) - (\mu_h + \delta)I,  \\
    \dot{R} &= \gamma_0 A + f(I) - \mu_h R, \\
    \dot{U} &= \Lambda_m - g(U,A,I) - (\mu_m + h(I))U,  \\
    \dot{V} &= g(U,A,I) - (\mu_m + h(I))V,
    \end{aligned}
\end{equation}
where $f(I)$ denotes the recovery rate of symptomatic individuals and is defined in equation~\eqref{recovery}, $g(U,A,I)$ represents the infection term in the mosquito population and is given in equation~\eqref{inf}, and $h(I)$ denotes the fogging intervention, as defined in equation~\eqref{fog}.

\subsection{Model parameters}
Our model's parameter values are based on well-established epidemiological, biological, and demographic data. These parameters are fixed based on demographic reports or previous studies. This section provides biological motivation and literature references for the parameters that are used in our model.

System~\eqref{model} discusses dengue transmission by considering demographic effects, namely the recruitment rate through newborns, denoted by $\Lambda_h$, and the natural death rate, denoted by $\mu_h$. The total human population is given by $\dfrac{dN_h}{dt} = \Lambda_h - \mu_h N_h - \delta I.$ If $\delta = 0$, and assuming that the total human population is constant, then we have $\Lambda_h = \mu_h N_h$. From the literature, $\mu_h$ is approximately given as the reciprocal of the life expectancy of the population. In Indonesia, which is taken as our study case, the life expectancy is approximately $64.7$ years. Hence, we have $\mu_h = \dfrac{1}{64.7 \times 365}\ \text{day}^{-1}.$ Assuming that the total population in Jakarta, Indonesia, is $11{,}249{,}585$, we then obtain $\Lambda_h = \dfrac{11{,}249{,}585}{64.7 \times 365}\ \text{human day}^{-1}.$

As in the human population, we also consider demographic effects in the mosquito population. The total mosquito population is given by $\dfrac{dN_m}{dt} = \Lambda_m - (\mu_m + h(I)) N_m$. Assuming that the total mosquito population is constant when fogging is not implemented, we then have $\Lambda_m = \mu_m N_m$. Based on physiological studies of mosquitoes and field surveys, adult mosquitoes can live approximately between 7.7 and 40 days in a supportive environment, such as an average temperature around $27^\circ$C \cite{Brady2013,Joy2012,Goindin2015}. Hence, we have $\mu_m \in \left[\dfrac{1}{40}, \dfrac{1}{7.7}\right]$. Assuming that the ratio between humans and mosquitoes, in the absence of fogging and dengue-induced mortality, is $1:2$ \cite{WIJAYA201983,Aguiar2022DengueModels}, we then obtain $\Lambda_m \in \left[\dfrac{2 \times 11{,}249{,}585}{40}, \dfrac{2 \times 11{,}249{,}585}{15}\right]$.

The transmission of dengue from human to mosquito and vice versa is governed by the product between the mosquito biting rate and transmission rate from human to mosquito and vice versa. This formulation explicitly captures both the frequency of human–mosquito contacts and the biological efficiency of virus transmission, and is widely adopted in dengue transmission models. In several literature, the biting rate of mosquitoes is between 0.26 to 0.67 per day \cite{Aguiar2022DengueModels,Ndii2016WolbachiaDengue,Mishra2020DengueADE}, and even higher up to 1 bite per day in endemic region \cite{EstevaVargas1998Dengue,YakobClements2013DengueModel}. On the other hand, the transmission rate from mosquito to human lies between 0 to 0.5 per day, while from human to mosquitoes between 0 and 0.75 \cite{Aguiar2022DengueModels,ALDILA2024115729,WOODALL201467,Rashkov01012021}. Therefore, we have $\beta_h \in [0, 0.5]$ and $\beta_{m1}, \beta_{m2} \in [0,0.75]$.

The parameter $p \in [0,1]$ denotes the probability, that a new infection is asymptomatic, i.e.~showing no symptoms. Biologically, this reflects heterogeneity in host immune responses, such as prior immunity, cross-reactive antibodies from previous dengue exposure, age, and genetic factors may suppress clinical symptoms despite successful viral infection. In~\cite{ALDILA2024115729}, the authors approximate 27\% of new infections do not show symptoms. However, it is possible that a prior asymptomatic case may develop symptoms at a later stage, which is denoted by the rate of $\alpha \in [0,1]$ in our model.

Biologically, asymptomatic individuals typically experience milder infection and recover without medical intervention, leading to a natural recovery process. Hospitalized patients benefit from intensive clinical care and close medical supervision, resulting in the fastest recovery rate. In contrast, symptomatic individuals treated at home recover more slowly due to limited medical support and delayed intervention compared to hospitalized individuals. Although asymptomatic individuals have a milder symptoms (almost none) due to a lower viral load, but it does not mean they will recover faster. Hence, we assume  $\gamma_0\leq\gamma_2 \leq \gamma_1$. From previous research by authors in \cite{ALDILA2024115729,WOODALL201467,XUE2021125742,Rashkov01012021}, dengue infected individuals can recover between 2 weeks to 3 days.

The hospital bed ratio per 1,000 people, as compiled by the World Health Organization (WHO) and reported in global health statistics, represents the average number of available hospital beds per 1,000 individuals in a given country or area \cite{hospitalBedsWHO}. In Indonesia, this ratio is approximately 1.2 beds per 1,000 people \cite{hospitalBedsIndonesia}, while in DKI Jakarta, the capital region, the ratio has been reported at around 2.6 beds per 1,000 people \cite{hospitalBedsJakarta}. However, it is important to note that not all hospital beds are allocated to dengue patients. Therefore, we assume that only a fraction of the total hospital capacity is effectively available for dengue treatment. In this study, this fraction is assumed to lie between 2\% and 10\%, which is consistent with realistic hospital resource allocation in the context of competing healthcare needs. Therefore, we have $C \in [584, 2924]$. The parameter $k$ in our model represent the hospital capacity threshold that will triggers fogging implementation in the field. In urban settings such as Jakarta as an example, intervention is typically initiated before the hospital is fully occupied. Hence, we assume that $k \in [0.1]$, and use $0.6$ as the baseline parameter to represent the medium response of the government. Smaller $k$ represent more responsive the government to respond dengue with fogging.

Last but not least, the death rate induced by dengue, which in our model is denoted by $\delta$. According to \cite{Schaefer2025DengueFever}, the case fatality rate (CFR) of severe dengue is reported to be 2–5\% with treatment and up to 10–20\% if untreated. On the other hand, global surveillance data in 2024 \cite{HAIDER2025107940} reported that the CFR reached 0.05–0.07\%. In Indonesia, the CFR reached 0.4–0.8\% \cite{KemenkesPusatKrisisDBD2026}. Since the infection period of dengue is between 7–14 days \cite{ALDILA2024115729}, then $\delta$ can be estimated as the ratio between the CFR and the infection period. Hence, $\delta \in \left[\dfrac{0.4\%}{14}, \dfrac{20\%}{7}\right] = [0.00028, 0.028]$.

\begin{landscape}
\begin{table}
    \centering
    \begin{tabular}{cp{10cm}ccp{2cm}}
        \hline Par. & Description & Interval & Ref.&Baseline\\ \hline
         $\Lambda_h$ & Recruitment rate of human (human $\cdot$ day $^{-1}$) & $ \left[\frac{11249585}{70 \times 365}, \frac{11249585}{50 \times 365} \right]$  & \cite{ALDILA2024115729,asamoah2022optimal,muharram2024indonesia,WOODALL201467,XUE2021125742,Rashkov01012021,Ghosh2020TwoStrainDengue}& $\frac{11249585}{64.7 \times 365}$\\
         $\Lambda_m$ & Recruitment rate of mosquito (mosquito $\cdot$ day $^{-1}$)& $\left[ \frac{2 \times 11249585}{40}, \frac{2 \times 11249585}{7.7} \right]$ & \cite{ALDILA2024115729,WIJAYA201983,Aguiar2022DengueModels,WOODALL201467,XUE2021125742,Rashkov01012021,Ghosh2020TwoStrainDengue} & $\frac{2 \times 11249585}{21}$ \\
         $\mu_h$ & Natural death rate of human ( day $^{-1}$)& $\left[\frac{1}{70 \times 365},\frac{1}{50 \times 365}\right]$ & \cite{ALDILA2024115729,asamoah2022optimal,muharram2024indonesia,WOODALL201467,XUE2021125742,Rashkov01012021,Ghosh2020TwoStrainDengue} &$\frac{1}{64.7 \times 365}$ \\
         $\mu_m$ & Natural death rate of mosquito ( day $^{-1}$)& $\left[\frac{1}{40}, \frac{1}{7.7}\right]$ & \cite{ALDILA2024115729,WOODALL201467,XUE2021125742,Rashkov01012021,Ghosh2020TwoStrainDengue} & $\frac{1}{21}$ \\
         $\beta_h$ & Transmission rate from mosquito to human ( (mosquito $\cdot$ day) $^{-1}$)& [0,1] & \cite{ALDILA2024115729,WOODALL201467,Rashkov01012021} & $\frac{0.472}{11249585}$\\
         $\beta_{m1}$ & Transmission rate from asymptomatic individual to mosquito ( (human $\cdot$ day) $^{-1}$) & [0,1] & \cite{ALDILA2024115729,MAIER201715,ANGGRIANI201962,BOCK2019108219} & $\frac{0.063}{11249585}$\\
         $\beta_{m2}$ & Transmission rate from symptomatic individual to mosquito ( (human $\cdot$ day) $^{-1}$) & $\beta_{m2}\geq \beta_{m1}$ & \cite{ALDILA2024115729} & $\frac{5\times0.063}{11249585}$\\
         $p$ & Fraction of new infected human develop to asymptomatic class & $[0,1]$ & \cite{ALDILA2024115729}  & 0.269\\
         $\alpha$ & Transition rate from asymptomatic to symptomatic class ( day $^{-1}$) & [0,1] & \cite{ALDILA2024115729,BOCK2019108219, Knerer2020DengueThailand,CHAMPAGNE20191} & 0.093\\

         $\gamma_0$ & Recovery rate of asymptomatic individual ( day $^{-1}$) & $\left[\frac{1}{14},\frac{1}{3}\right]$ & \cite{ALDILA2024115729,WOODALL201467,XUE2021125742,Rashkov01012021} & $\frac{1}{14}$ \\
         $\gamma_1$ & Recovery rate of symptomatic individual due to hospitalization ( day $^{-1}$) & $\left[\frac{1}{14},\frac{1}{3}\right]$  & \cite{ALDILA2024115729,WOODALL201467,XUE2021125742,Rashkov01012021} & $\frac{1}{10}$\\
         $\gamma_2$ & Recovery rate of symptomatic individual due to home treatment ( day $^{-1}$) & $\left[\frac{1}{14},\frac{1}{3}\right]$ & \cite{ALDILA2024115729,WOODALL201467,XUE2021125742,Rashkov01012021} & $\frac{1}{12}$\\

         $k$ & The proportion of hospital capacity at which fogging is initiated. & $(0,1]$ & Assumption & 0.6\\
         $C$ & The maximum hospital capacity for dengue patients (human) & [584,2924] & \cite{hospitalBedsJakarta} & 1500 \\
         $\eta$  & Fogging rate ( day $^{-1}$) & [0,1] & \cite{NAALY2024100159,MAYEDA2026508} &0.01\\
         $\delta$ & Death rate induced by dengue & [0.00028, 0.028]& \cite{Schaefer2025DengueFever,HAIDER2025107940,KemenkesPusatKrisisDBD2026} & $0.0015$\\
         \hline
    \end{tabular}
    \caption{Model parameters of model in~\eqref{model}.} 
    \label{tab:placeholder}
\end{table}
\end{landscape}

\section{Model analysis}\label{Sec3}
From the model description in the previous section, we can divide our model into three cases, namely when $I(t)< kC$, $kC\leq I(t)<C$, and when $I(t)>C$. Here below we give a mathematical analysis for each model regarding their existence of equilibria, the basic reproduction number, and the local stability criteria for each equilibrium point.

\subsection{Case when $I(t)< kC$}
When the number of infected individuals are less than the first threshold ($kC$, with $k \in (0,1)$), model~\eqref{model} read as follows:
\begin{eqnarray}\label{modela}
    \dot{S} &=& \Lambda_h - \beta_h S V - \mu_h S, \nonumber \\
    \dot{A} &=& p \beta_h S V - \left(\alpha + \gamma_0 + \mu_h\right)A, \nonumber \\
    \dot{I} &=& (1-p)\beta_h S V + \alpha A - \gamma_1 I - (\mu_h + \delta)I, \nonumber \\
    \dot{R} &=& \gamma_0 A + \gamma_1 I - \mu_h R, \\
    \dot{U} &=& \Lambda_m - \beta_{m1}UA - \mu_m U, \nonumber \\
    \dot{V} &=& \beta_{m1}UA - \mu_m V. \nonumber
\end{eqnarray}
For any non-negative initial condition of $S, A, I, R, U,$ and $V$, it can be shown that the direction of vector field in the boundary region of $\mathbb{R}_+^6$ always has an inward direction. This means that all solutions will always remain positive for any $t>0$.

Model~\eqref{modela} has two types of equilibria points. The first equilibrium point is the dengue-free equilibrium point, which is given by:
\begin{equation}\label{dfe}
    \mathcal{E}_0 = \left( S^0,A^0,I^0,R^0,U^0,V^0\right) = \left(\frac{\Lambda_h}{\mu_h}, 0, 0, 0, \frac{\Lambda_m}{\mu_m}, 0\right).
\end{equation}
The second equilibrium point is the endemic equilibrium point, which is given by:
\begin{subequations}
\begin{align}
    \mathcal{E}_1 &= \left(S^\dagger, A^\dagger, I^\dagger, R^\dagger, U^\dagger, V^\dagger\right), \label{Eq:E1}
\intertext{where}
    S^\dagger &= S^0 \cdot \frac{\mathcal{R}_0^2 \mu_h\mu_m + \Lambda_m \beta_h}{\mathcal{R}_0^2(\mu_h \mu_m + \Lambda_m \beta_h)}\;, \\
    A^\dagger &= (\mathcal{R}_0^2-1)\cdot \frac{\mu_h \mu_m^2}{\beta_{m1}(\mu_h \mu_m + \Lambda_m \beta_h)}\;, \\
    I^\dagger &= A^\dagger \cdot \frac{\alpha + (1-p)(\gamma_0+\mu_h)}{p (\gamma_1+\mu_h+\delta)}\;, \\
    R^\dagger &= A^\dagger\cdot \left( \frac{\gamma_0}{\mu_h} + \frac{\gamma_1}{\mu_h} \frac{\alpha + (1-p)(\gamma_0+\mu_h)}{p (\gamma_1+\mu_h+\delta)} \right) \;, \\
    U^\dagger &= U^0 \cdot \frac{\mu_h \mu_m + \Lambda_m \beta_h}{\mathcal{R}_0^2 \mu_h\mu_m + \Lambda_m \beta_h}\;, \\
    V^\dagger &= U^0 (\mathcal{R}_0^2-1)\cdot \frac{\mu_h \mu_m}{\mathcal{R}_0^2 \mu_h\mu_m + \Lambda_m \beta_h}\;.
\end{align}
\end{subequations}
Here, $\mathcal{R}_0 = \sqrt{p\, \dfrac{\Lambda_h \Lambda_m \beta_h \beta_{m1}}{\mu_h \mu_m^2 (\alpha+\gamma_0+\mu_h)}}$ denotes the basic reproduction number of the respected dengue model~\eqref{modela}. For $\mathcal{R}_0<1$ the infected compartments of the endemic equilibrium would be negative, i.e. they do not exist. Similarly, one can directly see, that $S^\dagger<S_0$ and $U^\dagger<U^0$ because the second factor is $<1$ for both terms in case of $\mathcal{R}_0>1$.

The basic reproduction number in the context of our dengue model represents the expected number of secondary cases of dengue caused by one primary case of infection during its infection period in a completely susceptible population. The basic reproduction number is taken from the spectral radius of the respected next-generation matrix, see \cite{van2002reproduction, diekmann2010construction}, of our model in~\eqref{modela}.
We note, that in the special case of $p=0$, we obtain $\mathcal{R}_0=0$; that aligns with the fact, that only for $p>0$, we get asymptomatic infections. And asymptomatic infections are required for the transmission cycle between humans and vectors.

Based on the above calculation, we have the following theorem to describe the existence criteria of the equilibria.
\begin{theorem}
    The dengue model~\eqref{modela} always has a dengue-free equilibrium. Furthermore, it always has a unique endemic equilibrium if $\mathcal{R}_0>1$ and no endemic equilibrium otherwise, where the basic reproduction number is given by $\mathcal{R}_0= \sqrt{p \dfrac{\Lambda_h \Lambda_m \beta_h \beta_{m1}}{\mu_h \mu_m^2 (\alpha+\gamma_0+\mu_h)}}$.
\end{theorem}

\textbf{Local Stability criteria of Dengue-free equilibrium and the endemic equilibrium.}

The establishment of the local stability of the Dengue-free equilibrium is presented in the following theorem.
\begin{theorem}
    Dengue-free equilibrium is locally asymptotically stable when $\mathcal{R}_0<1$ and unstable if $\mathcal{R}_0>1$.
\end{theorem}
\begin{proof}
After substituting the Dengue-free equilibrium $\mathcal{E}_0$ into the Jacobian matrix for the system~\eqref{modela}, we calculate the following six eigenvalues
\begin{align*}
    \lambda_{1,2}(\mathcal{E}_0) &= -\mu_h\;,\\
    \lambda_3(\mathcal{E}_0) &= -\mu_m\;,\\
    \lambda_4(\mathcal{E}_0) &= -(\gamma_1+\mu_h+\delta)\;, \\
    \lambda_{5,6}(\mathcal{E}_0) &= -\frac{1}{2} \left( -(\alpha+\gamma_0+\mu_h+\mu_m) \pm \sqrt{(\alpha+\gamma_0+\mu_h-\mu_m)^2+4p\beta_h \beta_{m1}\frac{\Lambda_h\Lambda_m}{\mu_h\mu_m}} \right)\;.
\end{align*}
It is easy to see, that all eigenvalues are negative provided $p\beta_h \beta_{m1}\frac{\Lambda_h\Lambda_m}{\mu_h\mu_m} < (\alpha+\gamma_0+\mu_h)\mu_m$, i.e.~if $\mathcal{R}_0<1$.

\end{proof}

The following theorem offers an approach for analyzing the local stability of the endemic equilibrium of the system \ref{modela}.

\begin{theorem}
The endemic equilibrium of~\eqref{modela} exhibits local asymptotic stability if $\mathcal{R}_0>1$.
\end{theorem}

\begin{proof}
By substituting the endemic equilibrium, $\mathcal{E}_1$, into the Jacobian matrix for the system, we first find three eigenvalues which they are $-\mu_{h}, -(\gamma_{1}+\mu_{h}+\delta), -\mu_{m}$. Clearly, they are negative, to check for the others, we analyze the following reduced characteristic polynomial that can be obtained after some lengthy but straightforward calculations
\begin{equation}
    \label{E:red-char-poly}
    \lambda^3+a_1 \lambda^2+a_2 \lambda+a_3=0\;,
\end{equation}
where 
\begin{align*}
    a_1 &= (\alpha+\gamma_0+\mu_m+2\mu_h)+ \mu_h (\mathcal{R}_0^2-1) \left[\frac{\beta_h \Lambda_m}{\mathcal{R}_0^2 \mu_h \mu_m + \beta_h \Lambda_m} + \frac{\mu_m^2}{\mu_h \mu_m + \Lambda_m \beta_h}\right] \;, \\
    a_2 &= (\alpha+\gamma_0+\mu_h) \frac{(\mathcal{R}_0^2-1) \mu_h \mu_m^2+\mu_h(\mu_h \mu_m + \beta_h \Lambda_m)}{\mathcal{R}_0^2 \mu_h \mu_m + \beta_h \Lambda_m}
    + (\mathcal{R}_0^2-1) \frac{\mu_h \mu_m \beta_h \Lambda_m}{\mu_h \mu_m + \beta_h \Lambda_m}\;, \\
    a_3 &= (\mathcal{R}_0^2-1)\mu_h \mu_m (\alpha + \gamma_0 + \mu_h) \;.
\end{align*}
It is obvious to see, that all three coefficients $a_1, a_2, a_3>0$, if $\mathcal{R}_0>1$. Moreover,
\begin{equation*}
    a_1 a_2 > (\alpha+\gamma_0+\mu_h) \dot (\mathcal{R}_0^2-1) \frac{\mu_h \mu_m \beta_h \Lambda_m}{\mu_h \mu_m + \beta_h \Lambda_m} > (\alpha+\gamma_0+\mu_h) \dot (\mathcal{R}_0^2-1) \frac{\mu_h \mu_m \beta_h \Lambda_m}{ \beta_h \Lambda_m} = a_3\;.
\end{equation*}
Here, we have just used the part $(\alpha+\gamma_0+\mu_h)$ from the $a_1$--factor and the last term in the $a_2$--factor, ignoring all the other, positive, terms.
Hence, all conditions of the Routh--Hurwitz criterion \cite{Bavafa2017}, are satisfied and we can conclude, that all roots of the cubic equation~\eqref{E:red-char-poly} have negative real part, i.e.~the endemic equilibrium $\mathcal{E}_1$ is locally asymptotically stable when $\mathcal{R}_0 > 1$.
\end{proof}

\subsection{Case when $kC \leq I(t) < C$}

When the number of symptomatic infected individuals satisfies $kC \le I(t)<C$, the fogging intervention is active, i.e.\ $h(I)=\eta$, while hospital capacity is not yet exceeded, hence $f(I)=\gamma_1 I$. Moreover, since $I(t)<C$, the mosquito infection term remains $g(U,A,I)=\beta_{m1}UA$. Therefore, system~\eqref{model} reduces to
\begin{eqnarray}\label{modelb}
\dot{S} &=& \Lambda_h - \beta_h S V - \mu_h S, \nonumber\\
\dot{A} &=& p \beta_h S V - \left(\alpha + \gamma_0 + \mu_h\right)A, \nonumber\\
\dot{I} &=& (1-p)\beta_h S V + \alpha A - \gamma_1 I - (\mu_h + \delta)I, \nonumber\\
\dot{R} &=& \gamma_0 A + \gamma_1 I - \mu_h R, \nonumber\\
\dot{U} &=& \Lambda_m - \beta_{m1}UA - (\mu_m+\eta)\,U, \nonumber\\
\dot{V} &=& \beta_{m1}UA - (\mu_m+\eta)\,V.
\end{eqnarray}

As in the previous case, positivity of solutions follows by checking that the vector field on each coordinate hyperplane points inward, standard for compartmental ODE models \cite{hethcote2000mathematics,smith1995theory}. Local existence and uniqueness hold in this regime because the right-hand side of \eqref{modelb} is locally Lipschitz (polynomial) in $(S,A,I,R,U,V)$, hence Picard--Lindel\"of applies \cite{coddington1955theory}.

The endemic equilibrium for the case $kC\le I \le C$ can be obtained by replacing $\mu_m$ with $\mu_m + \eta$ and $\mathcal{R}_0$ with $\mathcal{R}_{0\eta}$ in $\mathcal{E}_1$, see Eq. \eqref{Eq:E1}, where
\begin{equation}
    \mathcal{R}_{0\eta} = \frac{1}{1+\eta/\mu_m} \mathcal{R}_{0}
\end{equation}
is the basic reproduction number of the respected dengue model in~\eqref{modelb}. It is obvious to see, that $\mathcal{R}_{0\eta}<\mathcal{R}_0$ for $\eta>0$.

\begin{theorem}
    The endemic equilibrium of model \eqref{modelb} exhibits local asymptotic stability if the basic reproduction number, $\mathcal{R}_{0\eta}>1$.
\end{theorem}

\subsection{Case when $I(t)\geq C$}
For $I(t) \geq C$, where symptomatic infections surpass the capacity threshold, the fogging intervention is operative, i.e.~$h(I) = \eta$ and the hospital capacity is exceeded. Accordingly, the recovery function is defined as $f(I) = \gamma_1 C + \gamma_2 (I - C)$, and the mosquito transmission dynamics are governed by $g(U, A, I) = \beta_{m1} UA + \beta_{m2} U(I - C)$. Under these conditions, system \eqref{model} reads as

\begin{eqnarray}
\dot{S} &=& \Lambda_h - \beta_h S V - \mu_h S, \nonumber\\
\dot{A} &=& p \beta_h S V - \left(\alpha + \gamma_0 + \mu_h\right)A, \nonumber\\
\dot{I} &=& (1-p)\beta_h S V + \alpha A - \gamma_1 C-\gamma_2(I-C) - (\mu_h + \delta)I, \label{modelc} \\
\dot{R} &=& \gamma_0 A + \gamma_1 C+\gamma_2(I-C) - \mu_h R, \nonumber \\
\dot{U} &=& \Lambda_m - \beta_{m1}UA-\beta_{m2}U(I-C) - (\mu_m+\eta)\,U, \nonumber\\
\dot{V} &=& \beta_{m1}UA+\beta_{m2}U(I-C) - (\mu_m+\eta)\,V. \nonumber
\end{eqnarray}

Solving for the endemic equilibrium $\mathcal{E}^\ddagger=(S^\ddagger,A^\ddagger,I^\ddagger,R^\ddagger,U^\ddagger,V^\ddagger)$, we obtain
\begin{subequations}\label{end3}
\begin{align}
    A^\ddagger &= \frac{p\beta_h}{\alpha+\gamma_0+\mu_h} S^\ddagger V^\ddagger \\
    (I-C)^\ddagger &= \frac{\beta_h}{\gamma_2+\delta+\mu_h} \left[ 1-p + \frac{\alpha p}{\alpha+\gamma_0+\mu_h}\right] S^\ddagger V^\ddagger - \frac{\gamma_1+\delta+\mu_h}{\gamma_2+\delta+\mu_h} C \\
    R^\ddagger &= \lambda_R S^\ddagger V^\ddagger + \frac{\gamma_1-\gamma_2}{\mu_h} \frac{\delta+\mu_h}{\gamma_2+\delta+\mu_h} C \\
    U^\ddagger &= \frac{\Lambda_m}{B_0 S^\ddagger V^\ddagger + B_1} \stackrel{!}{=} \frac{\Lambda_m}{\mu_m+\eta} - V^\ddagger \label{end3d}\\
    S^\ddagger V^\ddagger &= \frac{\Lambda_h V^\ddagger}{\beta_h V^\ddagger+\mu_h}
\intertext{where}
    B_0 &= \frac{p \beta_h \beta_{m1}}{\alpha+\gamma_0+\mu_h} + \frac{\beta_h \beta_{m2}}{\gamma_2+\delta+\mu_h} \left[ 1-p + \frac{\alpha p}{\alpha+\gamma_0+\mu_h} \right] \nonumber \\
    B_1 &= -K C + \mu_m + \eta \nonumber \\
    \lambda_R &= \frac{\beta_h}{\mu_h} \left( \frac{\gamma_0}{\alpha+\gamma_0+\delta} p + \frac{\gamma_2}{\gamma_2+\delta+\mu_h} \left[ 1-p+ \frac{\alpha p}{\alpha+\gamma_0+\mu_h} \right] \right)\;. \nonumber
\end{align}
\end{subequations}
with $K = \beta_{m2}\dfrac{\gamma_1+\delta+\mu_h}{\gamma_2+\delta+\mu_h}$. Note, that $B_0, \lambda_R>0$ but $B_1$ gets negative for $C$ large.

The two equations for $U^\ddagger$ render the following quadratic equation
\begin{equation}\label{Fx}
    \mathcal{F}(x) = a_2 x^2 + a_1 x + a_0 = 0
\end{equation}
where $a_0 =\Lambda_m\mu_h \left( 1- \frac{B_1}{\mu_m+\eta}\right) >0$, $a_1 = \Lambda_m\beta_h + \mu_h B_1 - \frac{\Lambda_m}{\mu_m +\eta}a_2$ and $a_2= B_0\Lambda_h + \beta_h B_1$. We observe, that $a_2<0$ is equivalent to $C>C_2$, where
\[
C_2:= \dfrac{B_0\Lambda_h+\beta_h(\mu_m+\eta)}
{\beta_h K}.
\]
The solutions of the quadratic equation~\eqref{Fx} correspond to possible endemic equilibria $V^\ddagger$. Feasible solutions have to satisfy $U^\ddagger >0$ or equivalently $V^\ddagger < \dfrac{\Lambda_m}{\mu_m+\eta}$. Now, we distinguish the following cases:
\begin{enumerate}
\item If $a_2<0$, then  $\mathcal{F}(0)=a_0>0$ and $\mathcal{F}\left(\frac{\Lambda_m}{\mu_m+\eta}\right) >0$, so there exists no endemic equilibrium $0\le V^\ddagger<\dfrac{\Lambda_m}{\mu_m+\eta}$.

\item If $a_2>0$ and $a_1>-2\sqrt{a_0a_2}$ corresponds to the discriminant of $\mathcal{F}(x)$ being negative. Hence, the equation $\mathcal{F}(x)=0$ has no real solution and, consequently, no endemic equilibrium exists.

\item If $a_2>0$ and $a_1<-2\sqrt{a_0a_2}$, then the discriminant is positive. For the corresponding endemic equilibria to be biologically feasible, all components given in~\eqref{end3} must be positive. In particular, we have to ensure that $V^\ddagger<\dfrac{\Lambda_m}{\mu_m+\eta}$. Hence, we require
\[
\mathcal{F}'\left(\dfrac{\Lambda_m}{\mu_m+\eta}\right)>0.
\]
A direct calculation shows,  that this condition holds if and only if $C<C_3$, where
\[
C_3
=
\frac{1}{K}
\frac{
B_0\Lambda_h\Lambda_m
+\mu_h(\mu_m+\eta)^2
+2\beta_h\Lambda_m(\mu_m+\eta)
}{
\beta_h\Lambda_m+\mu_h(\mu_m+\eta)
}.
\]
In addition, the positivity condition for the infected population requires $I^\ddagger>C$. This condition is satisfied if $C<C_4$, where
\[
C_4
:=
\frac{\beta_h}{\gamma_2+\delta+\mu_h}
\left[
1-p+\frac{\alpha p}{\alpha+\gamma_0+\mu_h}
\right]
\frac{\Lambda_hV^\ddagger}{\beta_hV^\ddagger+\mu_h}
\frac{\beta_{m2}}{K-\beta_{m2}}.
\]
The condition $a_1>2\sqrt{a_0 a_2}$ leads to a second order polynomial $g(C)=g_2 C^2+g_1 C+g_0$ with respect to $C$. Although the coefficients are lengthy expressions, one can see, that $g_2>0$ and $g_1<0$; the sign of $g_0$ depends in a complicated and non--obvious way on all parameters of the problem. Nevertheless, the condition $a_1>2\sqrt{a_0 a_2}$ is satisfied, if $C<C_1^{(1)}$ or $C>C_1^{(2)}$ for some $C_1^1 < C_1^2$. In this situation, we get two possible endemic equilibria form equation ~\eqref{Fx}.

Consequently, two biologically feasible endemic equilibria exist provided that
\[
C<C_2,\qquad
C<C_3,\qquad
C<C_4,
\]
and
\[
C<C_1^{(1)}
\qquad\text{or}\qquad
C>C_1^{(2)}.
\]
\end{enumerate}

Based on the numerical values given in Table~\ref{tab:placeholder}, we have $C_1^{(1)}=3103, C_1^{(2)}=5629, C_2=1.7\times 10^6, C_3=3.4\times10^6,$ and $C_4=25351$, then we have $\min\{C_1,C_2,C_3,C_4\} = 3103.$ Since our parameter values in Table~\ref{tab:placeholder} with $C=1500$ are satisfying the second case in our endemic existence condition, so there are two endemic equilibria $ E^\ddagger_1 = \left( S_1^\ddagger, A_1^\ddagger, I_1^\ddagger-C, U_1^\ddagger, V_1^\ddagger\right)
    = \left(386736, 752, 2991, 18566007, 28348 \right)$ and $E^\ddagger_2  = \left(7160289, 283, 7, 18593779, 576 \right)$.
On the other hand, if we choose $C>\min\{C_1,C_2,C_3,C_4\},$ for an example $C=3104$, then we will have no positive roots of $\mathcal{F}(x),$ and hence no endemic equilibrium in the regime-3, where $I(t)>C$. However, we may still have the endemic equilibrium whether it is in the first regime where $I(t)< kC$ or in the second regime where $kC\leq I(t)<C$.

\section{Numerical investigation of dengue transmission model}\label{Sec4}

In this section we will present a detailed numerical study of the dynamical response of the dengue transmission model proposed in this work, in the context of piecewise-smooth dynamical systems \cite{librochamp}. To this end, we will employ numerical continuation (path-following) techniques, which provide a structured framework for examining how a model responds to variations in its parameters, with particular emphasis on identifying qualitative transitions in behavior, such as bifurcations. For investigating periodic solutions in piecewise-smooth systems using continuation techniques, several dedicated computational tools are available, including SlideCont \cite{slidecont}, TC-HAT \cite{tchat} and COCO \cite{dankowicz2013}. In this study, COCO is selected as the primary tool for the numerical analysis of the dengue transmission model. The following section presents a detailed description of the mathematical formulation required to conduct this investigation using COCO.

\subsection{The dengue transmission model as a piecewise-smooth dynamical system}

In order to investigate the dengue transmission model \eqref{model} by means of numerical continuation in COCO, the governing equations must be cast into the framework of piecewise-smooth dynamical systems, see below
\begin{equation}
	\dot{z}(t)=\left(\setstretch{1.25}\begin{array}{c}
		\Lambda_h - \beta_h S(t) V(t) - \mu_h S(t)\\
		p \beta_h S(t) V(t) - \left(\alpha + \gamma_0 + \mu_h\right)A(t)\\
		(1-p)\beta_h S(t) V(t) + \alpha A(t) - F_{\text{\tiny
				DEN}}(t) - (\mu_h + \delta)I(t)\\		
		\Lambda_m - G_{\text{\tiny
				DEN}}(t) - (\mu_m + H_{\text{\tiny
				DEN}}(t))U(t)\\
		G_{\text{\tiny
				DEN}}(t) - (\mu_m + H_{\text{\tiny
				DEN}}(t))V(t)			
	\end{array}\right)=:f_{\text{\tiny
			DEN}}\left(z(t),\rho,H_{\mathrm{C}},H_{\mathrm{kC}}\right),\label{eq-hybsys}
\end{equation}
where the vector
$z=(S,A,I,U,V)^T\in\re^5_{\geq0}$ denotes the state variables of the system and $\rho\in\re^{16}_{>0}$ the parameter vector specified in Table \ref{tab:placeholder}. Notice that in system \eqref{eq-hybsys} the $R$-compartment has been omitted since it has no dynamical influence on the remaining compartments, it only serves to monitor the evolution of recovered individuals. The functions $F_{\text{\tiny
DEN}}$, $G_{\text{\tiny DEN}}$ and $H_{\text{\tiny DEN}}$ appearing in \eqref{eq-hybsys} are defined as (see \eqref{recovery}--\eqref{fog})
\begin{align}
	F_{\text{\tiny DEN}}(t)&=\gamma_1 I(t)(1-H_{\mathrm{C}})+\left(\gamma_1 C + \gamma_2 (I(t) - C)\right)H_{\mathrm{C}},\nonumber\\
	G_{\text{\tiny DEN}}(t)&=\beta_{m1} U(t) A(t)+\beta_{m2} U(t) (I(t) - C)H_{\mathrm{C}},\label{eq-fum}\\
	H_{\text{\tiny DEN}}(t)&=\eta H_{\mathrm{kC}},\nonumber
\end{align}
$t\geq0$. Here, $H_{\mathrm{C}}$ and $H_{\mathrm{kC}}$ are discrete switching variables that encode the distinct operation regimes of the system, and are given by
\begin{align}
	H_{\mathrm{C}}&=\begin{cases}
		1, & I>C,\mbox{ }\mbox{ }\mbox{(dengue infections above hospital capacity)},\\
		0, & I\leq C,\mbox{ }\mbox{ }\mbox{(dengue infections below hospital capacity)},
	\end{cases}\label{eq-C}\\
	H_{\mathrm{kC}}&=\begin{cases}
		1, & I\geq kC,\mbox{ }\mbox{ }\mbox{(fogging applied)},\\
		0, & I<kC,\mbox{ }\mbox{ }\mbox{(no fogging)}.
	\end{cases}\label{eq-kC}
\end{align}
The switching variables defined above encode the distinct operation regimes of the system, as listed in Table \ref{tab-modes}. Here, the fogging takes place whenever $H_{\mathrm{kC}}=1$, and no fogging otherwise. Also, it is important to note the not biologically feasible scenario where $H_{\mathrm{kC}}=0$ and $H_{\mathrm{C}}=1$, which should be discarded for the numerical implementation.

\begin{table}[htpb!]
	\caption{Operation modes of the dengue transmission model and the corresponding
		values of the discrete variables $H_{\mathrm{C}}$ and $H_{\mathrm{kC}}$
		defined in Eqs.\ \eqref{eq-C}, \eqref{eq-kC}.} \label{tab-modes} \vspace{-10pt}
	\begin{center}
		\small \setstretch{1.2}\begin{tabular}{|c|c|c|}\hline
			\textbf{Operation mode} & $H_{\mathrm{kC}}$ & $H_{\mathrm{C}}$\\
			\hline
			Dengue infections $I$ below $kC$ & 0 & 0\\
			\hline	Not biologically feasible & 0 & 1\\
			\hline	Dengue infections $I$ between $kC$ and $C$ & 1 & 0\\
			\hline	Dengue infections $I$ above $C$ & 1 & 1\\
			\hline
		\end{tabular}
	\end{center}
\end{table}

\subsection{Solution measures}\label{sec-measures}

In the present numerical investigation, we will discover the possibility of oscillating solutions for the dengue transmission model \eqref{eq-hybsys}. For such periodic responses, to assess performance and identify favorable operating regimes, suitable quantitative indicators will be introduced. Suppose that $(S(t),A(t),I(t),U(t),V(t))$, $t\geq0$, is a bounded $T_{0}$-periodic solution of \eqref{eq-hybsys}. Under this assumption, we define
\begin{equation}
	\label{eq-M_AVcost}
	\mbox{Cost}=\frac{1}{T_{0}}\int_{0}^{T_{0}}H_{\text{\tiny DEN}}(s)\,ds,
\end{equation}
which gives the average cost of the fogging policy within one period. Similarly, we introduce
\begin{equation}
	\label{eq-MDARK}
	A_{\text{\tiny Frac,\%}}=\frac{\int_{0}^{T_{0}}A(s)\,ds}{\int_{0}^{T_{0}}I(s)\,ds+\int_{0}^{T_{0}}A(s)\,ds}\times100.
\end{equation}
This indicator represents the percentage of asymptomatic individuals infected with dengue in relation to the whole infected population. And finally, we consider the quantity
\begin{equation}
	\label{eq-MIpeak}
	I_{\text{\tiny Peak}}=\max_{0\leq t\leq
	T_{0}}I(t),
\end{equation}
which gives the peak value of the infected compartment within one period. The quantities defined above provide meaningful information about the dynamical behavior of the dengue transmission model \eqref{eq-hybsys} and the efficacy of the fogging policy. They will be used later to identify optimal operating conditions under parameter variations.

\subsection{Numerical investigation of equilibria}

To begin our numerical study, let us consider the transient response of the dengue transmission model \eqref{eq-hybsys}, computed for the parameter values given in Table \ref{tab:placeholder}, see Fig.\ \ref{fig-sol-ini-end}. In this case, the system operates in the regime $kC\leq I(t)\leq C$, with $H_{\mathrm{kC}}=1$ and $H_{\mathrm{C}}=0$, meaning that the fogging policy is activated and the dengue infections are within the hospital capacity. As the transient decays, the system settles down to an endemic equilibrium, which will be used as starting point for the numerical continuation analysis.

\begin{figure}[H]
	\centering\psfrag{#}{\scriptsize\hspace{-1pt}$\times$}\psfrag{t}{\large$t$ \scriptsize[days]}\psfrag{I}{$I(t)$}\psfrag{S}{$S(t)$}\psfrag{A}{$A(t)$}\psfrag{V}{$V(t)$}\psfrag{kC}{\small $kC$}\psfrag{C}{\small $C$}
	\includegraphics[width=\textwidth]{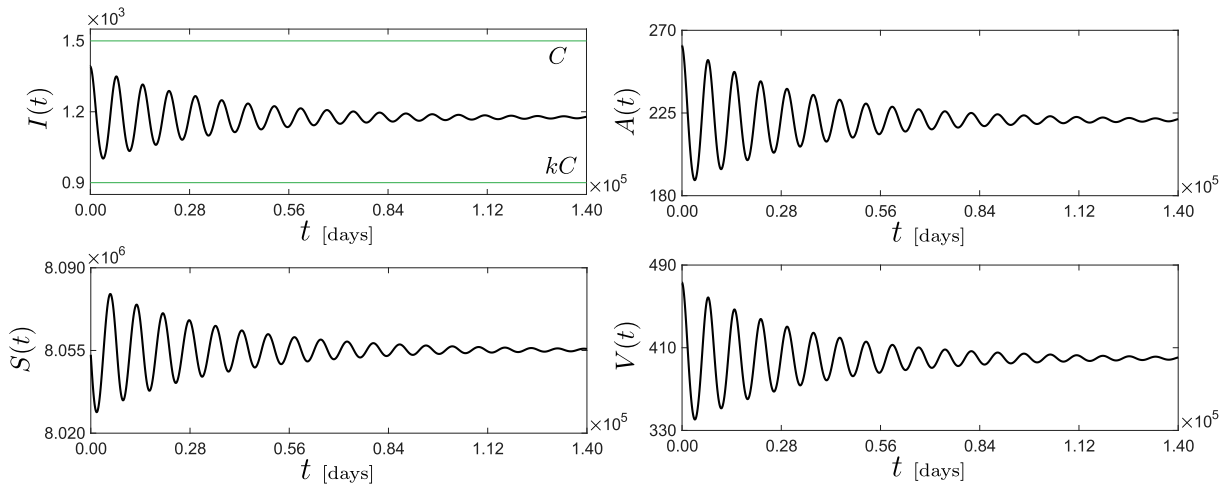}
	\caption{Transient response of the dengue transmission model \eqref{eq-hybsys}, computed for the parameter values given in Table \ref{tab:placeholder}. In this picture, the horizontal green lines denote the critical thresholds $I=kC$ and $I=C$.}\label{fig-sol-ini-end}
\end{figure}

The numerical continuation investigation will consider first the mosquito-to-human transmission rate $\beta_{h}$ as main bifurcation parameter, which gives a measure related to human exposure to the dengue carriers. The result is presented in Fig.\ \ref{fig-bif-diags-eq}(a). Here we see the continuation of the starting equilibrium (marked P1) found in Fig.\ \ref{fig-sol-ini-end} with respect to $\beta_{h}$. As $\beta_{h}$ decreases, the $I$ component of the equilibrium decreases as well, until it reaches the critical value $I=kC$, for $\beta_{h}\approx3.8417\times10^{-8}$. Here, a boundary-equilibrium bifurcation BE2 takes place \cite{librochamp}. At this value, there is a discontinuity in the vector field, produced by the fogging function $H_{\text{\tiny DEN}}$ defined in \eqref{eq-fum} and \eqref{eq-kC}. The equilibrium branch for $I<kC$ starts at the boundary-equilibrium bifurcation BE1, found for $\beta_{h}\approx2.6208\times10^{-8}$. As $\beta_{h}$ reduces, so does the infected compartment $I$, until the critical value BP ($\beta_{h}\approx2.0537\times10^{-8}$) is found, where a forward bifurcation takes place. A this point, a branch of endemic and disease-free equilibria meet, with the corresponding stability change for the latter.

Now let us investigate the fate of the endemic equilibrium P1 for increasing values of $\beta_{h}$. As can be expected, the infected compartment $I$ increases with $\beta_{h}$, until a nonsmooth fold bifurcation (NSF1, $\beta_{h}\approx4.7047\times10^{-8}$) is detected \cite{librochamp}. This is the moment where the infections have reached the hospital capacity $I=C$, which marks a critical level for the health system. At this point, a branch of unstable endemic equilibria is born, with a nonsmooth intersection with the stable branch. For $\beta_{h}$ values beyond NSF1, oscillatory behavior with high infections levels is detected, produced by the Hopf bifurcation H ($\beta_{h}\approx7.6763\times10^{-7}$) visible in Fig.\ \ref{fig-bif-diags-eq}(b). Also, we notice the presence of a smooth fold point F ($\beta_{h}\approx5.5212\times10^{-9}$), which has no significant dynamical consequences for the system behavior, as it simply adds another real eigenvalue to the right-half plane.

\begin{figure}[H]
	\centering
	\psfrag{#}{\scriptsize\hspace{-1pt}$\times$}\psfrag{a}{\large(a)}\psfrag{b}{\large(b)}\psfrag{p}{\large(c)}\psfrag{d}{\large(d)}\psfrag{Bh}{\large$\beta_{h}$}\psfrag{et}{\large$\eta$}\psfrag{kC}{\small $kC$}\psfrag{C}{\small $C$}\psfrag{MI}{\large$\left.I\right|_{\footnotesize\mbox{eq}}$}\psfrag{R1}{$\left.I\right|_{\footnotesize\mbox{eq}}\leq kC$}\psfrag{R2}{\small$kC<\left.I\right|_{\footnotesize\mbox{eq}}\leq C$}\psfrag{R3}{$I_{\text{\tiny Peak}}>C$}\psfrag{L1}{\large$\ell_{1}$}
	\psfrag{L2}{\large$\ell_{2}$}
	\includegraphics[width=\textwidth]{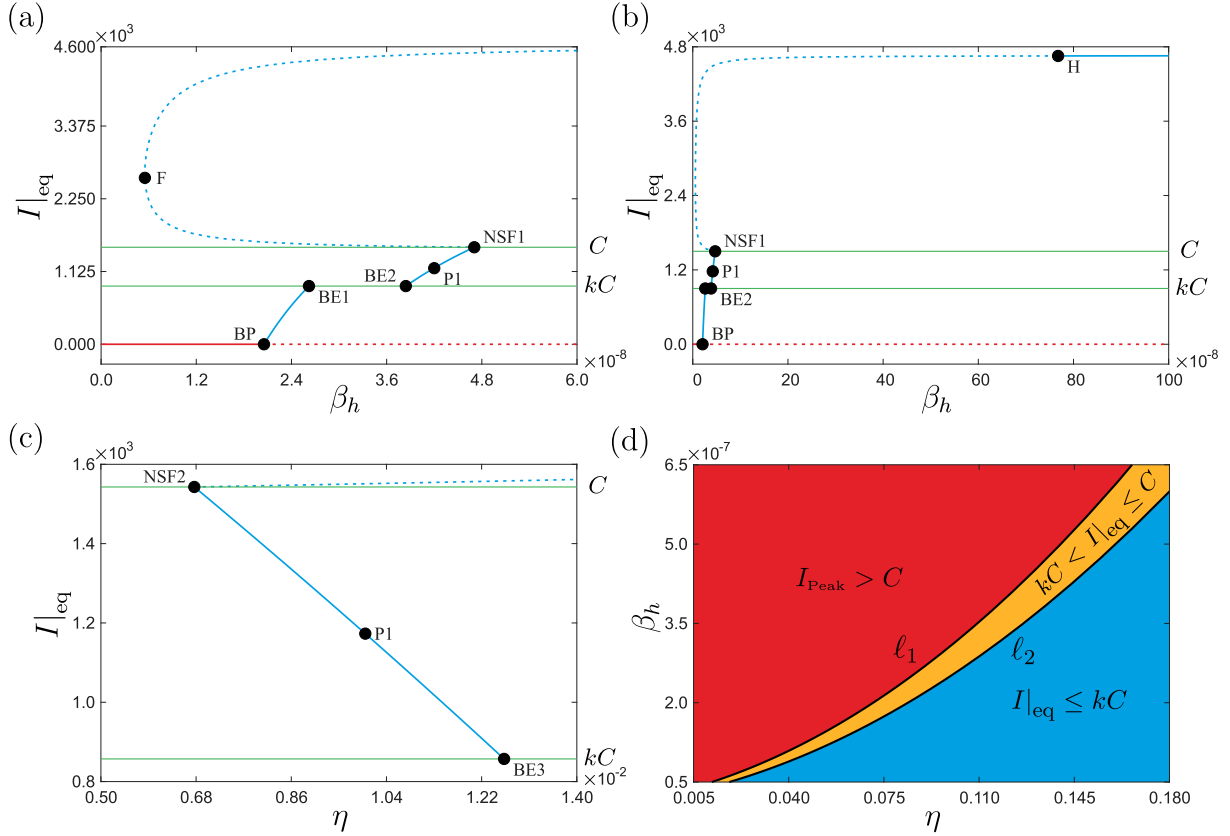}
	\caption{(a) One-parameter continuation of the endemic equilibrium encountered in Fig.\ \ref{fig-sol-ini-end} (marked P1 in the diagram) with respect to the mosquito-to-human transmission rate $\beta_{h}$. (b) Same bifurcation diagram as in panel (a), showing a larger parameter range. (c) Continuation of the endemic equilibrium with respect to the fogging rate $\eta$. (d) Two-parameter continuation of the bifurcation points NSF2 (curve $\ell_{1}$) and BE3 (curve $\ell_{2}$) with respect to $\eta$ and $\beta_{h}$, dividing the parameter space into three regions, as indicated in the diagram. Here, $I_{\text{\tiny Peak}}$ refers to the peak $I$-value of a periodic solution (see Fig.\ \ref{fig-sol-ini-per}). In panels (a)--(c), branches of stable solutions are plotted with solid lines, while dashed lines denote instability. The detected bifurcation points are located as follows: branching point (BP, $\beta_{h}\approx2.0537\times10^{-8}$), boundary equilibrium (BE1, $\beta_{h}\approx2.6208\times10^{-8}$; BE2, $\beta_{h}\approx3.8417\times10^{-8}$; BE3, $\eta\approx1.2625\times10^{-2}$), nonsmooth fold (NSF1, $\beta_{h}\approx4.7047\times10^{-8}$; NSF2, $\eta\approx6.7663\times10^{-3}$), fold (F, $\beta_{h}\approx5.5212\times10^{-9}$) and Hopf (H, $\beta_{h}\approx7.6763\times10^{-7}$).}\label{fig-bif-diags-eq}
\end{figure}

Next, we will consider another parameter for the continuation study, namely, the fogging rate $\eta$. Our main interest here is to try to understand its effect on the system and verify the effectiveness of the proposed disease control strategy. The result is displayed in Fig.\ \ref{fig-bif-diags-eq}(c), which shows the numerical continuation of the starting equilibrium labeled P1 with respect to $\eta$. The observed decreasing behavior of the $I$ compartment confirms the effectiveness of the fogging action. As $\eta$ varies, again two nonsmooth bifurcations of equilibria are found, detected at $\eta\approx6.7663\times10^{-3}$ (nonsmooth fold NSF2) and $\eta\approx1.2625\times10^{-2}$ (boundary equilibrium bifurcation BE3). These points  mark again critical biological limit cases, namely, the activation of the fogging policy (BE3, when $I=kC$) and the reaching of the hospital capacity (NSF2, when $I=C$). Let us now investigate how these critical points are affected when two parameter are allowed to vary simultaneously. Specifically, we will carry out the two-parameter continuation of the critical points BE3 and NSF2 with respect to the fogging rate $\eta$ and mosquito-to-human transmission rate $\beta_{h}$. The result is presented in Fig.\ \ref{fig-bif-diags-eq}(d). The resulting curves divide the parameter space into three biologically relevant regions: $I\leq kC$ (endemic equilibria with no fogging), $kC<I\leq C$ (endemic equilibria with fogging) and $I_{\text{\tiny Peak}}>C$ (oscillatory behavior with high infections levels, see next section). The increasing pattern of the boundary curves $\ell_{1}$ and $\ell_{2}$ reveal the fact that higher mosquito-to-human transmission rates $\beta_{h}$ demand higher fogging rates $\eta$, in order to bring the system to a biologically acceptable scenario, from a health policy perspective.

\begin{figure}[H]
	\centering\psfrag{#}{\scriptsize\hspace{-1pt}$\times$}\psfrag{a}{\large(a)}\psfrag{b}{\large(b)}\psfrag{p}{\large(c)}\psfrag{d}{\large(d)}\psfrag{t}{\large$t$ \scriptsize[days]}\psfrag{I}{$I(t)$}\psfrag{h}{$H_{\text{\tiny DEN}}(t)$}\psfrag{kC}{\small $kC$}\psfrag{C}{\small $C$}\psfrag{et}{\small $\eta$}
	\includegraphics[width=\textwidth]{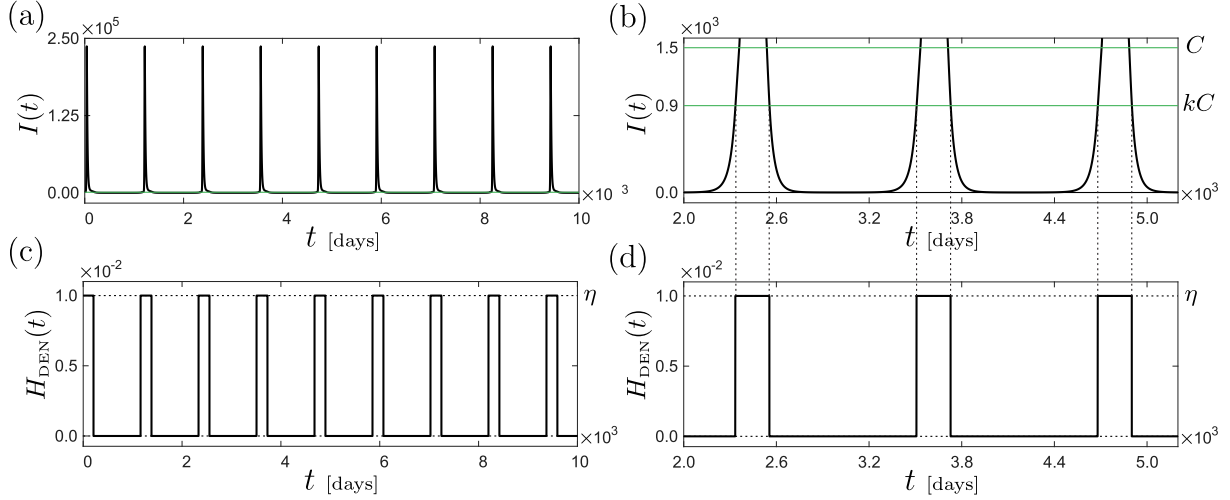}
	\caption{Periodic response of the dengue transmission model \eqref{eq-hybsys}, computed for the parameter values given in Table \ref{tab:placeholder}, with $\beta_{h}=9.3\times10^{-7}$. Panels (b) and (d) display the same solutions as those in (a) and (c), for a narrower time window.}\label{fig-sol-ini-per}
\end{figure}

\subsection{Numerical investigation of oscillatory responses}

As already discussed in the previous section, the dengue transmission model \eqref{eq-hybsys} is able to show oscillatory behavior, produced by the Hopf bifurcation detected in Fig.\ \ref{fig-bif-diags-eq}(b). A reference periodic solution is presented in Fig.\ \ref{fig-sol-ini-per}, computed for the parameter values given in Table \ref{tab:placeholder}, with $\beta_{h}=9.3\times10^{-7}$. In this figure, we can observe the behavior of the fogging signal $H_{\text{\tiny DEN}}(t)$, which presents a rectangular-wave pattern, due to the repetitive crossing of the infected compartment with the critical boundary $I=kC$, hence activating and deactivating the fogging action. This can be closer seen at the blow-up shown in panel (b) and (d) of Fig.\ \ref{fig-sol-ini-per}.

\begin{figure}[H]
	\centering
	\psfrag{#}{\scriptsize\hspace{-1pt}$\times$}\psfrag{a}{\large(a)}\psfrag{b}{\large(b)}\psfrag{c}{\large(c)}\psfrag{d}{\large(d)}\psfrag{e}{\large(e)}\psfrag{f}{\large(f)}\psfrag{k}{\large$k$}\psfrag{et}{\large$\eta$}\psfrag{MIpeak}{\large$I_{\text{\tiny Peak}}$}
	\psfrag{MDARK}{\large$A_{\text{\tiny Frac,\%}}$}\psfrag{Cost}{\large Cost}
	\includegraphics[width=\textwidth]{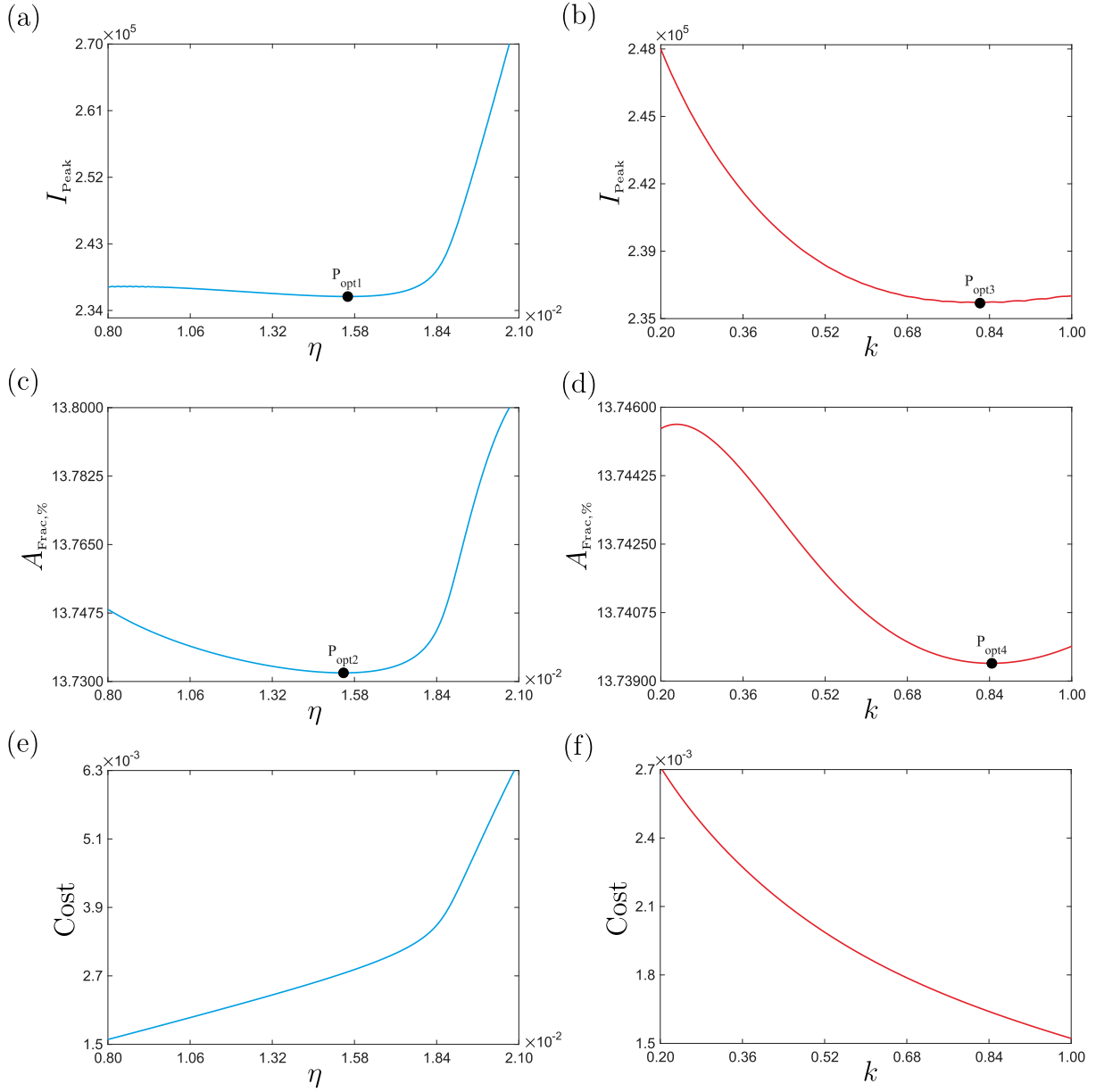}
	\caption{One-parameter continuation of the periodic response displayed in Fig.\ \ref{fig-sol-ini-per}, with respect to the asymptomatic fogging rate $\eta$ (panels (a), (c) and (e), in blue) and fogging proportion threshold $k$ (panels (b), (d) and (f), in red).  In these diagrams, the vertical axes show the behavior of the peak $I$-value $I_{\text{\tiny Peak}}$, fraction of asymptomatic infections $A_{\text{\tiny Frac,\%}}$ and fogging cost (see Section \ref{sec-measures}). Local minima of the monitor functions are located at $\eta\approx1.5583\times10^{-2}$ (P$_{\scriptsize\mbox{opt1}}$), $\eta\approx1.5454\times10^{-2}$ (P$_{\scriptsize\mbox{opt2}}$), $k\approx0.8215$ (P$_{\scriptsize\mbox{opt3}}$) and $k\approx0.8445$ (P$_{\scriptsize\mbox{opt4}}$).}\label{fig-bif-diags-per}
\end{figure}

Let us now investigate how the main fogging policy parameters $\eta$ and $k$ affect the behavior of the reference oscillatory response discussed above. Specifically, we will carry out the numerical continuation of the piecewise-smooth periodic solution shown in Fig.\ \ref{fig-sol-ini-per}, with respect to the fogging rate $\eta$ and fogging proportion threshold $k$, one parameter at the time. The result is plotted in Fig.\ \ref{fig-bif-diags-per}. These diagrams present the behavior of the solutions measures defined in Section \ref{sec-measures}, namely, peak $I$-value $I_{\text{\tiny Peak}}$, fraction of asymptomatic infections $A_{\text{\tiny Frac,\%}}$ and fogging cost. Interestingly, there is no monotonic relation between $I_{\text{\tiny Peak}}$, $A_{\text{\tiny Frac,\%}}$ with respect to the selected bifurcation parameters. In fact, one can identify optimal values of both $\eta$ and $k$ for which the aforementioned solution measures are locally minimized. The optimal values are as follows: $\eta\approx1.5583\times10^{-2}$ (P$_{\scriptsize\mbox{opt1}}$, panel (a)), $\eta\approx1.5454\times10^{-2}$ (P$_{\scriptsize\mbox{opt2}}$, panel (c)), $k\approx0.8215$ (P$_{\scriptsize\mbox{opt3}}$, panel (b)) and $k\approx0.8445$ (P$_{\scriptsize\mbox{opt4}}$, panel (d)). Noticeably, the optimal values minimizing  $I_{\text{\tiny Peak}}$ and $A_{\text{\tiny Frac,\%}}$ are different but very close to one another, indicating that there is a connection between the observed peak values of the infected compartment and the fraction of asymptomatic infections, showing the importance of considering this latter compartment for health monitoring strategies. On the other hand, panels (e) and (f) of Fig.\ \ref{fig-bif-diags-per} presents the behavior of the average fogging costs within one period of the oscillatory solution. As can be expected, higher applied fogging rates increase the cost, while low values of fogging proportion threshold $k$ augment such cost, since low $k$ means the fogging action is applied earlier. In the limit, for $k\to0$, we have that fogging is applied all the time, which is of course a nonviable scenario from an economic point of view. Besides, the optimal points P$_{\scriptsize\mbox{opt3}}$ and P$_{\scriptsize\mbox{opt4}}$ indicate that the resources have to be applied in a suitable manner, in order to achieve health policy goals in an efficient way.

\section{Discussion and Conclusions}\label{Sec5}
\subsection{Discussion}
Dengue remains an annual public health problem in many tropical countries, with recurring outbreaks indicating that current intervention strategies are not yet fully effective \cite{guzman2010dengue}. Two of the most widely used interventions are fumigation (fogging) and hospitalization of infected individuals \cite{oki2011optimal}. Many existing mathematical models incorporate these interventions \cite{aldila2023impact,rodrigues2010insecticide,rodrigues2013bioeconomic}, but typically without accounting for practical resource constraints, such as limited hospital bed capacity and reactive, threshold-triggered fogging schedules. In practice, however, hospital capacity is limited, and fogging is generally not applied continuously but rather triggered reactively once the number of reported infections reaches a certain threshold tied to the epidemic situation. This discrepancy between idealized model assumptions and real-world resource constraints motivates the need for a model that explicitly accounts for both limited hospital capacity and threshold-triggered fogging. To address this, we develop a non-smooth ordinary differential equation model of dengue transmission in which a hospital capacity $C$ and a fogging-triggering threshold $kC$, $k \in [0,1]$, jointly govern the dynamical system. The model distinguishes three regimes: when infections remain below the trigger threshold ($I < kC$), no fogging is implemented and all infected individuals can be hospitalized; once infections cross the threshold but remain within capacity ($kC \leq  I < C$), fogging is activated while hospitalization continues to isolate all cases; and when infections exceed hospital capacity ($I \geq  C$), fogging remains active, but the excess $I - C$ individuals cannot be admitted to the hospital and instead remain in the population, where they continue to transmit the disease and further endanger public health.

To understand the qualitative behavior of the system within each operating condition, we first analyze the model analytically in each of the three regimes separately, since the governing equations---and hence the relevant equilibria---differ across regimes due to the discontinuous nature of the fogging and hospitalization mechanisms. This regime-wise analysis provides the foundation for identifying how the disease dynamics change as the epidemic crosses the intervention thresholds $I=kC$ and $I=C$. In the first regime ($I<kC$), we establish the existence and local stability of both the disease-free and endemic equilibria. Consistent with standard epidemic model behavior, the disease-free equilibrium is stable when the basic reproduction number $\mathcal{R}_0<1$ and unstable when $\mathcal{R}_0>1$, while the endemic equilibrium exists only when $\mathcal{R}_0>1$ and is stable whenever it exists, indicating the occurrence of a transcritical bifurcation at $\mathcal{R}_0=1$. In the second regime ($kC<I<C$), the same qualitative structure is obtained; however, the activation of fogging, governed by the fogging rate $\eta$, reduces the reproduction number relative to the first regime, reflecting the impact of the intervention on disease transmission. The third regime ($I>C$) exhibits more complex behavior: since $I>C$ by definition, the disease-free equilibrium and its associated reproduction number are no longer epidemiologically relevant in this regime. Instead, we focus on the existence of endemic equilibria and find that two endemic equilibria emerge simultaneously through a saddle-node bifurcation. The stability of these equilibria, along with the switching behavior at the regime boundaries $I=kC$ and $I=C$, is investigated numerically in the following section.

 Numerical continuation confirms the analytical findings obtained in each regime and further reveals the rich dynamical structure induced by the switching thresholds $I=kC$ and $I=C$. Boundary-equilibrium bifurcations are detected precisely at these two critical values, marking the points at which the dynamical system transitions between the no-fogging, fogging, and hospital-overflow regimes; their presence confirms that the switching thresholds are not merely technical artifacts of the piecewise formulation but act as genuine organizing centers of the dynamical system, with direct epidemiological meaning as early-warning and capacity-saturation indicators. Beyond these boundary bifurcations, continuation with respect to the transmission rate reveals a Hopf bifurcation once infections exceed hospital capacity, giving rise to sustained periodic solutions with recurrent high-infection peaks. Such oscillatory behavior implies that outbreaks may recur periodically rather than settle into a single endemic level, considerably complicating control efforts, as interventions must then be sustained or timed appropriately to suppress recurring waves rather than a single peak. Moreover, the continuation results reveal a fold point $F$, close to the transcritical threshold, at which two additional endemic equilibria emerge while $\mathcal{R}_0<1$; although these equilibria are unstable and therefore do not compromise the guarantee of disease elimination once $\mathcal{R}_0<1$, their presence indicates that the equilibrium landscape near the epidemic threshold is more intricate than a standard transcritical exchange, underscoring the importance of a careful stability analysis near this threshold to ensure that public health guidance based solely on $\mathcal{R}_0$ correctly reflects the underlying dynamics.

The non-monotonic dependence of $I_{Peak}$ and $A_{Frac,\%}$ on both the threshold proportion $k$ and the fogging rate $\eta$ can be explained by the interaction between the timing of the fogging response and the population dynamics of the mosquito vector. For the threshold parameter $k$, a smaller value triggers fogging earlier and keeps it active over a substantially longer fraction of the epidemic cycle, since $I$ takes longer to fall back below a lower threshold $kC$. Because the human and mosquito compartments in model~\eqref{model} respond to fogging with an inherent structural delay---through the incubation and infectious periods on the human side, and the recruitment and infection processes on the mosquito side---an early and near-continuous fogging response does not necessarily synchronize with the natural period of the underlying epidemic cycle. Because the intervention is triggered by the epidemic state rather than by a fixed schedule, a poorly chosen threshold level can fail to align with the epidemic's intrinsic growth and decay pattern, so that the suppression of transmission is delayed rather than eliminated. This allows susceptible individuals to accumulate during the intervention phase and produces a larger rebound once the threshold condition is no longer met. This behavior is consistent with previous modeling work showing that the timing, rather than merely the intensity, of insecticide fogging is a critical determinant of its effectiveness in reducing dengue incidence \cite{oki2011optimal}. A similar but mechanistically distinct explanation applies to the fogging rate $\eta$: once $I$ exceeds the activation threshold $kC$, a larger $\eta$ suppresses the mosquito population more strongly for as long as fogging remains active, but this suppression ends abruptly once $I$ falls below $kC$, at which point the mosquito population recovers due to the continuous larval recruitment term $\Lambda_m$, which is unaffected by $h(I)$. This is consistent with field evidence showing that fogging predominantly targets adult mosquitoes while having limited effect on the larval and pupal stages, so that adult mosquito populations frequently rebound within days after fogging ceases \cite{mani2005efficacy}, and that routine chemical fogging is largely ineffective at breaking the reproductive cycle of gravid female \textit{Aedes} mosquitoes, leaving the immature population essentially unaffected \cite{chua2005effect}. Consequently, larger values of $\eta$ can produce a stronger but shorter-lived suppression, followed by a more pronounced resurgence once the intervention is deactivated and the accumulated susceptible population becomes available for transmission again. Taken together, these results indicate that the effectiveness of threshold-triggered fogging depends not only on the intensity of the intervention but also on how well its activation and deactivation are synchronized with the intrinsic epidemic cycle and the biology of vector recovery, an insight that is directly relevant for the design of reactive vector-control policies.

\subsection{Conclusion and Future Works}

This study proposed a non-smooth mathematical model of dengue transmission that explicitly incorporates limited hospital capacity and a threshold-triggered fogging policy, two operational features of real dengue control programs that are often idealized away in existing models. Our analytical and numerical results show that these capacity and threshold constraints are not merely operational details but genuine drivers of the disease dynamics: they generate boundary-equilibrium bifurcations, give rise to periodic outbreaks once hospital capacity is exceeded, and produce a non-monotonic response of key public health indicators to the fogging rate and activation threshold. From a public health standpoint, these findings suggest that an effective dengue control strategy requires careful, joint calibration of hospital capacity, fogging intensity, and response timing, rather than treating early or intensive intervention as inherently optimal; poorly timed or excessively delayed fogging can instead prolong transmission or trigger a larger resurgence in cases.

Several directions remain open for future research. On the analytical side, a formal classification of the criticality of the Hopf bifurcation, together with a Filippov-based sliding-mode analysis of the switching manifolds $I=kC$ and $I=C$, would provide a more complete characterization of the system's qualitative behavior. On the modeling side, incorporating seasonal forcing in mosquito population dynamics, stochastic effects in transmission and case reporting, and spatial heterogeneity in hospital capacity across regions would bring the model closer to real-world dengue transmission settings. Finally, formulating the fogging and hospitalization strategy as an optimal control problem, and calibrating the model against real epidemiological and hospitalization data, would strengthen its applicability as a decision-support tool for dengue control policy.

\section*{CRediT authorship contribution statement}
D.A., J.P.C., A.G., T.G., B.G.: Conceptualization, investigation, methodology, validation, writing – original draft, writing – review and editing. D.A., B.G.: Funding acquisition. D.A. : Project administration. D.A., A.G., T.G., B.G.: formal analysis.  J.P.C.: Software, Visualization.

\section*{Declaration of competing interest}
The authors declare that they have no known competing financial interests or personal relationships that could have appeared to influence the work reported in this paper.

\section*{Funding}
D.A. is funded by the Indonesian Endowment Fund for Education (LPDP) on behalf of the Indonesian Ministry of Higher Education, Science and Technology and managed under the EQUITY Program (Contract No. 4302/B3/DT.03.08/2025 and 573/PKS/R/UI/2025). B.G. gratefully acknowledges the support of the Johanna Quandt Young Academy (JQYA)-Goethe University Frankfurt Fellowship.

\section*{Data availability}
Not applicable.

\bibliography{denguedipobg}

\end{document}